\documentclass[12pt,reqno]{amsart}
\usepackage{amsmath, amssymb, amsfonts, xcolor}
 \usepackage{hyperref}
\numberwithin{equation}{section}
\usepackage{amsthm}
\newtheorem{thm}{Theorem}[section]

\newtheorem{prop}[thm]{Proposition}
\newtheorem{lem}[thm]{Lemma}
\newtheorem{conj}[thm]{Conjecture}
\newtheorem{dfn}[thm]{Definition}

\newtheorem{remark}[thm]{Remark}

\usepackage[a4paper, top = 1.2in, bottom = 1.2in, left = 1.2in, right = 1.2in]{geometry}

\numberwithin{equation}{section}

\newcommand{\F}{\mathbb{F}}
\newcommand{\N}{\mathbb{N}}
\newcommand{\Q}{\mathbb{Q}}

\newcommand{\Z}{\mathbb{Z}}

\newcommand{\mcO}{\mathcal{O}}

\newcommand{\mfm}{\mathfrak{m}}
\newcommand{\mfn}{\mathfrak{n}}

\newcommand{\mfp}{\mathfrak{p}}
\newcommand{\mfq}{\mathfrak{q}}
\newcommand{\mfP}{\mathfrak{P}}

\newcommand{\GL}{\mathrm{GL}}

\newcommand{\Cl}{\mathrm{Cl}}

\newcommand{\Gal}{\mathrm{Gal}}

\def\1{1\!\!1}
\newcommand{\mrm}[1]{\mathrm{#1}}

\title[On the solutions of equations of signature $(p,p,3)$ over $K$]{Asymptotic solutions of the generalized Fermat-type equation of signature $(p,p,3)$ over totally real number fields}

\author[S. Sahoo]{Satyabrat Sahoo}
\address[S. Sahoo]{Yau Mathematical Sciences Center, Tsinghua University, 
	Beijing 100084, China.}
\email{satyabrat.sahoo.94@gmail.com}

\author[N. Kumar]{Narasimha Kumar}
\address[N. Kumar]{Department of Mathematics, Indian Institute of Technology Hyderabad, Kandi, Sangareddy 502285, INDIA.}
\email{narasimha@math.iith.ac.in}

\date{\today}

\begin{document}
	\maketitle
	\begin{abstract}
In this article, we study the asymptotic solutions of the generalized Fermat-type equation of signature $(p,p,3)$ over totally real number fields $K$, i.e., $Ax^p+By^p=Cz^3$ with prime exponent $p$ and $A,B,C \in \mcO_K \setminus \{0\}$. For certain class of fields $K$, we prove that $Ax^p+By^p=Cz^3$ has no asymptotic solutions over $K$ (resp., solutions of certain type over $K$) with restrictions on $A,B,C$ (resp., for all $A,B,C \in \mcO_K \setminus \{0\}$). Finally, we present several local criteria over $K$.
\end{abstract}

\maketitle
  \section{Introduction}
   The study of Diophantine equations is an interesting and extremely fascinating area in number theory. The most important example of the Diophantine equation is the Fermat equation $x^n+y^n=z^n$. In 1637, Fermat claimed that the equation $x^n+y^n=z^n$ with positive integers $n \geq 3$ has no non-trivial coprime integer solutions. This was settled affirmatively by Wiles. The main inputs in the proof were the modularity of semi-stable elliptic curves $E/\Q$ (cf.~\cite{W95},~\cite{TW95}), irreducibility of the mod $p$ Galois representations $\bar{\rho}_{E,p}$ attached to $E/\Q,p$ (cf. \cite{M78}), and the level-lowering theorem of Ribet for $\bar{\rho}_{E,p}$ (cf. \cite{R90}).
   
   Since then, there has been a lot of progress in understanding  the solutions of the generalized Fermat-type equation
   \begin{equation}
\label{generalized Fermat eqn}
 	Ax^p+By^q=Cz^r
\end{equation}
with coprime integers $A,B,C$ and integers $p,q,r \geq 2$ with $\frac{1}{p} +\frac{1}{q}+ \frac{1}{r} <1$. We call $(p,q,r)$ as the signature of \eqref{generalized Fermat eqn}. In~\cite{DG95}, Darmon and Granville showed that the generalized Fermat-type equation~\eqref{generalized Fermat eqn}, with fixed $A,B,C, p,q,r$, has only finitely many non-trivial coprime integer solutions.
 
\subsection{Literature Survey:}
Throughout this article, $K$ denotes a totally real number field, and $\mcO_K$ denotes the ring of integers of $K$. Let $n \in \N$ and $p \in \mathbb{P}:= \mrm{Spec}(\Z)$. In~\cite{DM97}, Darmon and Merel proved that the equation $x^n+y^n=z^3$ with exponent $n \geq 3$ has no non-trivial coprime integer solutions. In \cite{BVY04}, Bennett, Vatsal, and Yazdani studied the integer solution of the generalized Fermat-type equation of signature $(n,n,3)$, i.e., $Ax^n+By^n=Cz^3$, and proved that the equation $x^n+y^n=pz^3$ has no coprime integer solutions with $|xy|>1$ and prime $n>p^{4p^2}$. Recently, in \cite{M22}, Mocanu studied the asymptotic solutions (of certain type over $K$) for the equation $x^p+y^p=z^3$ with exponent $p$ (cf. Definition~\ref{asymptotic solution} for the definition of the asymptotic solutions). More recently, in~\cite{IKO23}, I\c{s}ik, Kara, and \"Ozman studied the asymptotic solution of certain type over number fields $F$ for the equation $x^p+y^p=z^3$ with exponent $p$, by assuming two modularity conjectures (cf. \cite[Conjectures 2.2, 2.3]{IKO23}) when the narrow class number of $F$ is $1$.

In this article, we generalize the work of \cite{M22} and study the asymptotic solutions of the generalized Fermat-type equation of signature $(p,p,3)$, i.e., $Ax^p+By^p=Cz^3$ with exponent $p$ over $K$, where $A,B,C \in \mcO_K\setminus \{0\}$. More precisely,
\begin{itemize}
	\item In Theorem~\ref{main result1 for x^2=By^p+Cz^p Type I}, we prove that for certain class of fields $K$, the equation $Ax^p+By^p=Cz^3$ with prime exponent $p$ has no asymptotic solution in $W_K$ (cf. Definition~\ref{def for W_K} for $W_K$ and Definition~\ref{asymptotic solution} for the asymptotic solution).
	
	\item In Theorem~\ref{main result1 for Ax^p+By^p=qz^3 over K}, we prove that for certain class of fields $K$, the equation $Ax^p+By^p=Cz^3$ with prime exponent $p$ has no asymptotic solution in $\mcO_K^3\setminus S$ for some explicit set $S \subseteq \mcO_K^3$, whenever $3$ is inert in $K$, $v_\mfP(A)=1$, $ v_\mfP(B) \in \{0,2\}$ for $\mfP |3$, and $C \in \mcO_K^*$ or $C=uq$ for $u \in \mcO_K^*$ and $q \in \mathbb{P} \setminus \{ 3\}$ (cf. \S \ref{section for main result over K} for $S$). In fact, $S = \emptyset$ if $A,B,C \in \Z \setminus \{0\}$ (cf.
	Proposition~\eqref{prop for S is empty}).
\end{itemize}
The proofs of Theorems~\ref{main result1 for x^2=By^p+Cz^p Type I},~\ref{main result1 for Ax^p+By^p=qz^3 over K} depend upon certain explicit bounds on the solutions of the equation
\begin{equation}
	\label{keyequation}
\alpha +\beta =\gamma^3
\end{equation}
with $\alpha, \beta \in \mcO_{S_K^{\prime}}^\ast, \gamma \in \mcO_{S_K^{\prime}}$
 (cf. \S \ref{notations section for Ax^p+By^p=Cz^3} for the definitions of
$S^\prime_K, \mcO_{S_K^{\prime}}, \mcO_{S_K^{\prime}}^\ast$).
Finally, we provide several local criteria of $K$ for Theorem~\ref{main result1 for x^2=By^p+Cz^p Type I}.

\subsection{Literature for the equations of signature $(p,p,p)$ and $(p,p,2)$}
In~\cite{FS15}, Freitas and Siksek studied the asymptotic solutions of the Fermat equation $x^p+y^p=z^p$ with exponent $p$ over $K$. In \cite{D16}, Deconinck extended the work of~\cite{FS15} to the generalized Fermat equation $Ax^p+By^p=Cz^p$ with $A, B, C$ odd. In \cite{KS23}, we studied the asymptotic solutions of the equation $x^p+y^p=2^rz^p$ with exponent $p$ and $r \in \N$. In general, there has been a lot of progress, which is summarized below, for the generalized Fermat-type equation of signature $(n,n,2)$ over $K$. Let $r,n \in \N$ and $p\in \mathbb{P}:= \mrm{Spec}(\Z)$.
	\begin{table}[h!]
		\centering
	\begin{tabular}{|c|c|c|}  
		\hline  
		Fermat-type equation & $\Q$-solutions  & asymptotic $K$-solutions \\ \hline
		$x^n+y^n=z^2$ & \cite{DM97} &~\cite{IKO20}, \cite{M22}, \cite{KS23} for $n=p$\\ \hline
		$x^p+2^ry^p=z^2$& \cite{I03}, \cite{S03}& \cite{KS}\\ \hline 
		$x^p+2^ry^p=2z^2$& \cite{I03}& \cite{KS}\\ \hline  
		$Ax^n+By^n=Cz^2$& \cite{BS04} & \cite{KS} for $n=p$ and $C=1$ or $2$\\ \hline
	\end{tabular}  
\end{table}
%
%
	
	\subsection{Preliminaries}
	\label{section for preliminary}
	 Let $P :=\mrm{Spec}(\mcO_K)$ and $\mfn$ be an ideal of $\mcO_K$.
	 Let $E/K$ be an elliptic curve of conductor $\mfn$. 
	 Let $\bar{\rho}_{E,p} : G_K:=\Gal(\overline{K}/K) \rightarrow \mathrm{Aut}(E[p]) \simeq \GL_2(\F_p)$ be the residual Galois representation of $G_K$, induced by the action of $G_K$ on the $p$-torsion of $E$. For any $\mfp \in P$, let $I_\mfp$ be the inertia subgroup of $G_K$ at $\mfp$. For any Hilbert modular newform $f$ over $K$ of parallel weight $k$, level $\mfn$ with coefficient field $\Q_f$ and any  $\omega \in \mrm{Spec}(\mcO_{\Q_f})$, let $\bar{\rho}_{f, \omega}: G_K \rightarrow \GL_2(\F_\omega)$ be the residual Galois representation attached to $f, \omega$. The following conjecture is a generalization of the Eichler-Shimura theorem over $\Q$ (cf. \cite[Conjecture 1]{FS15}).
	\begin{conj}[Eichler-Shimura]
		\label{ES conj}
		Let $f$ be a Hilbert modular newform over $K$ of parallel weight $2$, level $\mfn$, and with coefficient field $\Q_f= \Q$. Then, there exists an elliptic curve $E_f /K$ with conductor $\mfn$ having same $L$-function as $f$.
	\end{conj}
    	In~\cite[Theorem 7.7]{D04}, Darmon proved Conjecture~\ref{ES conj} over $K$, when 
    $[K: \Q] $ is odd or there exists some $\mfq \in P$ such that $v_\mfq(\mfn) = 1$. 
    In \cite[Corollary 2.2]{FS15}, Freitas and Siksek provided a partial answer to Conjecture~\ref{ES conj} in terms of the residual Galois representations $\bar{\rho}_{E,p}$ attached to $E$. 
%
For any $\mfq \in P$, let $\Delta_\mfq$ be the minimal discriminant of $E$ at $\mfq$. Let
 \begin{equation}
 	\label{conductor of elliptic curve}
 	\mfm_p:= \prod_{ p|v_\mfq(\Delta_\mfq), \ \mfq ||\mfn} \mfq \text{ and } \mfn_p:=\frac{\mfn}{\mfm_p}.
 \end{equation}
  For any rational prime $p$, let $\zeta_p$ be a primitive $p$-th root of unity. 
 We end this section with a result of Freitas and Siksek (cf.~\cite[Theorem 7]{FS15}) on the level-lowering of the residual Galois representations $\bar{\rho}_{E,p}$ over $K$.
 \begin{thm}
 	\label{level lowering of mod $p$ repr}
 	Let $E$ be an elliptic curve over $K$ of conductor
 	$\mfn$. Let $p$ be a rational prime. Suppose that the following conditions hold:
 	\begin{enumerate}
 		\item  For $p \geq 5$, the ramification index $e(\mfq /p) < p-1$ for all $\mfq |p$, and $\Q(\zeta_p)^+ \nsubseteq K$;
 		\item $E/K$ is modular and $\bar{\rho}_{E,p}$ is irreducible;
 		\item $E$ is semi-stable at all $\mfq |p$, and $p| v_\mfq(\Delta_\mfq)$ for all $\mfq |p$.
 	\end{enumerate}
 	Then there exists a Hilbert modular newform $f$ over $K$ of parallel weight $2$, level $\mfn_p$, and some prime $\omega$ of $\Q_f$ such that $\omega | p$ and $\bar{\rho}_{E,p} \sim \bar{\rho}_{f,\omega}$.
 \end{thm}

		\subsection{Our strategy:}
In this section, we discuss the modular method, due to Freitas and Siksek~\cite{FS15}, to prove Theorems~\ref{main result1 for x^2=By^p+Cz^p Type I},~\ref{main result1 for Ax^p+By^p=qz^3 over K}. The proof of Theorem~\ref{main result1 for x^2=By^p+Cz^p Type I} is fairly standard. The novelty of our work can be seen in the proof of Theorem~\ref{main result1 for Ax^p+By^p=qz^3 over K}, which is elaborated below.
\begin{enumerate}
	\item For any non-trivial primitive solution $(a, b, c)\in \mcO_{K}^3$ to the equation $Ax^p+By^p=Cz^3$, we attach a Frey elliptic curve $E:=E_{a,b,c}$ as in \eqref{Frey curve for x^2=By^p+Cz^p of Type I}. Then we prove the modularity of $E$ for $p \gg 0$, and that $E$ has semi-stable reduction at $\mfq$ away from $ S_K^{\prime}:= \{ \mfP \in P :\ \mfP|3ABC \}$ along with $p | v_\mfq(\Delta_E)$. Using~\cite[Theorem 2]{FS15 Irred}, the residual representation $\bar{\rho}_{E,p}$ is irreducible for $p \gg 0$.
	
	\item Let $ (a,b,c)\in \mcO_{K}^3 \setminus S$ (cf. \S\ref{section for main result over K} for the definition of $S$). By a result of Freitas and Siksek on the image of inertia at $\mfq \in P$ (cf. Lemma~\ref{criteria for potentially multiplicative reduction}) and a result of \cite{K90} (cf. Lemma~\ref{image of inertia comparision}) for $p \gg0$, we get that either $p | \#\bar{\rho}_{E,p}(I_\mfP)$ or $ \#\bar{\rho}_{E,p}(I_\mfP) \in \{3,6\}$ for $\mfP \in S_K:= \{ \mfP \in P :\ \mfP|3\}$ if $3$ is inert in $K$ and with some assumptions on $A,B,C$ (cf. Lemma~\ref{reduction on T and S x^2=By^p+Cz^p Type II over K}).
	
	\item Now, using the modularity of $E$, irreducibility of $\bar{\rho}_{E,p}$, semi-stable reduction of $E$ away from $S_K^{\prime}$ and a level-lowering result by ~\cite[Theorem 7]{FS15}, there exists an elliptic curve $E'/K$ having a non-trivial $3$-torsion point with $\bar{\rho}_{E,p} \sim \bar{\rho}_{E^\prime,p}$ for $p\gg 0$ (cf. Theorem~\ref{auxilary result x^2=By^p+2^rz^p over K}). 
	Again, using Lemmas~\ref{criteria for potentially multiplicative reduction}, ~\ref{image of inertia comparision}, we get either $v_\mfP(j_{E^\prime})<0$ or $v_\mfP(j_{E^\prime}) \equiv 0 \text{ or } 2 \pmod 3$ for $\mfP \in S_K$. Finally, using a technique of Mocanu in \cite{M22}, we relate $j_{E'}$ in terms of solutions of~\eqref{keyequation}, together with~\eqref{assumption for main result1 x^2=By^p+2^rz^p over K}, to get $v_\mfP(j_{E^\prime}) \geq0$ and  $v_\mfP(j_{E^\prime}) \equiv 1 \pmod 3$ for some $\mfP \in S_K$ to get a contradiction.
\end{enumerate}

	\section{Solutions of the equation $Ax^p+By^p=Cz^3$ over $W_K$}
	\label{notations section for Ax^p+By^p=Cz^3} 	
	In this section, we study the solutions of the following equation:
	\begin{equation}
		\label{Ax^p+By^p=Cz^3}
		Ax^p+By^p=Cz^3
	\end{equation} 
	with prime exponent $p\geq 3$ and $A,B,C \in \mcO_K\setminus \{0\}$.
	Recall that $P=\mrm{Spec}(\mcO_K)$. Let $S_K:= \{ \mfP \in P :\ \mfP|3\}$ and $S_K^{\prime}:= \{ \mfP \in P :\ \mfP|3ABC \}$.

	\begin{dfn}[Trivial solution]
		We say a solution $(a, b, c)\in \mcO_K^3$ to the equation \eqref{Ax^p+By^p=Cz^3} with exponent $p$ is trivial, if $abc=0$,
		otherwise non-trivial.
		We say $(a, b, c)\in \mcO_K^3$ is primitive if $a\mcO_K+b\mcO_K+c\mcO_K=\mcO_K$.
	\end{dfn}  
\begin{dfn}
\label{def for W_K}
Let $W_K$ be the set of all non-trivial primitive solutions $(a, b, c)\in  \mcO_K^3$ to the equation~\eqref{Ax^p+By^p=Cz^3} with exponent $p$ such that $\mfP |ab$ for every $\mfP \in S_K$.
\end{dfn}
	\begin{remark}
		\label{remark for W_K and W_K'}
		Let $\mfP \in S_K$. If $(a, b, c)\in W_K$ with $p > v_\mfP (C)$, then $\mfP$ divides exactly one of $a, b$. Otherwise $\mfP^p | Aa^p+Bb^p=Cc^3$ and since  $p >  v_\mfP (C)$, we get $\mfP |c$. This is a contradiction to the primitivity of $(a, b, c)$.
	\end{remark}

	\subsection{Main result}
	\label{section for main result of x^2=By^p+2^rz^p} 
	For any set $S \subseteq P$, let $\mcO_{S}:=\{\alpha \in K : v_\mfP(\alpha)\geq 0 \text{ for all } \mfP \in P \setminus S\}$ be the ring of $S$-integers in $K$ and $\mcO_{S}^*$ be the $S$-units of $\mcO_{S}$.  
    Let $\Cl_S(K):= \mrm{Cl}(K)/\langle [\mfP]\rangle_{\mfP\in S}$ and $\Cl_S(K)[n]$ be its $n$-torsion points, where $\mrm{Cl}(K)$ denotes the class group of $K$.
\begin{dfn}
	\label{asymptotic solution}
	We say a Diophantine equation $Ax^p+By^p=Cz^3$ with exponent $p$ has no asymptotic solution in a set $W \subseteq \mcO_K^3$, if there exists a constant $V_{K,A,B,C}>0$ (depending on $K,A,B,C$) such that for primes $p>V_{K,A,B,C}$, the equation $Ax^p+By^p=Cz^3$ with exponent $p$ has no non-trivial primitive solution in $W$.
\end{dfn}
    We now show that the equation~\eqref{Ax^p+By^p=Cz^3} with exponent $p$ has no asymptotic solution in $W_K$. More precisely,
	\begin{thm}
		\label{main result1 for x^2=By^p+Cz^p Type I}
		\label{main result1 for x^2=By^p+Cz^p Type II}
		Let $K$ be a totally real field with $\Cl_{S_K^{\prime}}(K)[3]=1$. Suppose for every solution $(\alpha, \beta, \gamma) \in \mcO_{S_K^{\prime}}^\ast \times \mcO_{S_K^{\prime}}^\ast \times \mcO_{S_K^{\prime}}$ to  $\alpha+\beta=\gamma^3$,
        there exists $\mfP \in S_K$ that satisfies 
		\begin{equation}
			\label{assumption for main result1 for x^2=By^p+Cz^p Type I}
	\left| v_\mfP \left(\alpha \beta^{-1}\right) \right| \leq 3v_\mfP(3).
		\end{equation}
		Then, the equation $Ax^p+By^p=Cz^3$ with exponent $p$ has no asymptotic solution in $W_K$.
\end{thm}

\begin{remark}
By~\cite[Theorem 39]{M22}, for any finite set $S \subseteq P$, the equation $\alpha+\beta=\gamma^3$ with $\alpha, \beta \in \mcO_{S}^\ast$ and $\gamma \in \mcO_{S}$ has only finitely many solutions up to the equivalence $\sim$, where $\sim$ is defined as follows: $(\alpha, \beta, \gamma) \sim (\alpha', \beta', \gamma')$ if there exists $\epsilon \in \mcO_{S}^\ast$ such that $\alpha= \epsilon^3 \alpha'$, $\beta= \epsilon^3 \beta'$ and $\gamma= \epsilon \gamma'$. So the hypothesis \eqref{assumption for main result1 for x^2=By^p+Cz^p Type I} in Theorem~\ref{main result1 for x^2=By^p+Cz^p Type I} needs to be checked only finitely many times.
\end{remark}
 For any number field $F$, let $h_F$ denote the class number of $F$. We say that $S \subseteq P$ is principal if $\mfP$ is principal for all $\mfP \in S$. If $3$ is inert in $K$, and either $A,B,C  \in \mcO_K^\ast$ or $A,B,C \in \mathbb{P}$ are inert in $K$, then $S_K^{\prime}$ is principal. In this case, we have $\Cl_{S^\prime_K}(K)= \Cl(K)$, and $\Cl(K)[3]=1$ is equivalent to $3 \nmid h_K$. Let $\zeta_3$ be a primitive cubic root of unity. The following proposition is a consequence of Theorem~\ref{main result1 for x^2=By^p+Cz^p Type I}, which will be useful in \S\ref{section for loc criteria}.
\begin{prop}
	\label{main result2 for x^2=By^p+Cz^p Type I}
	\label{main result2 for x^2=By^p+Cz^p Type II}
	Let $K$ be a totally real field such that $S_K^{\prime}=S_K$ is principal and $3 \nmid  h_K h_{K(\zeta_3)}$. Assume that $3$ is inert or totally ramified in $K$,
	and $S_K= \{\mfP\}$. Suppose for every solution $(\alpha, \gamma) \in \mcO_{S_K}^\ast \times \mcO_{S_K}$ to  $\alpha+1=\gamma^3$ with $v_\mfP(\alpha) \geq 0$ satisfies the inequality
	\begin{equation}
	\label{assumption for main result2 for x^2=By^p+Cz^p Type I}
v_\mfP(\alpha)\leq 3v_\mfP(3).
	\end{equation} 
  	Then, the equation $Ax^p+By^p=Cz^3$ with exponent $p$ has no asymptotic solution in $W_K$.
\end{prop}

	\subsection{Construction of Frey elliptic curves}
	For any non-trivial and primitive solution $(a, b, c)\in \mcO_K^3$ to the equation \eqref{Ax^p+By^p=Cz^3} with exponent $p$, the Frey curve $E:=E_{a,b,c}$ is given by
		\begin{equation}
			\label{Frey curve for x^2=By^p+Cz^p of Type I}
			E:=E_{a,b,c} : Y^2+3CcXY+C^2Bb^pY = X^3,
		\end{equation}
    	with $c_4=3^2C^3c(9Aa^p+Bb^p),\ \Delta_E=3^3AB^3C^8(ab^3)^p$ and $ j_E=3^{3} \frac{Cc^3(9Aa^p+Bb^p)^3}{AB^3(ab^3)^p}$, where $j_E$ (resp., $\Delta_E$) denote the $j$-invariant (resp., discriminant) of $E$. 
\subsection{Modularity of Frey elliptic curves}
    	We now use a modularity result of Freitas, Le Hung and Siksek (cf.~\cite[Theorem 5]{FLHS15}) to prove the modularity of the Frey curve $E:=E_{a,b,c}$ in ~\eqref{Frey curve for x^2=By^p+Cz^p of Type I} associated to $(a,b,c)\in W_K$ for primes $p \gg 0$.
	\begin{thm}
		\label{modularity of Frey curve x^2=By^p+Cz^p over W_K}
		Let K be a totally real number field. Then, there exists a constant $D:=D_{K,A,B,C}$ (depending on $K,A,B,C$) such that
		for any solution $(a,b,c)\in W_K$ to the equation \eqref{Ax^p+By^p=Cz^3}  with exponent $p>D$, the Frey curve $E:=E_{a,b,c}$ given in~\eqref{Frey curve for x^2=By^p+Cz^p of Type I} is modular.  
	\end{thm}
	
	\begin{proof}
    By~\cite[Theorem 5]{FLHS15}, there exist only finitely many elliptic curves  over $K$, up to $\bar{K}$-isomorphism, which are not modular. Let $j_1,\ldots,j_s \in K$ be the $j$-invariants of those elliptic curves.
		The $j$-invariant of the Frey curve $E$ is given by $ j_E=3^{3} \frac{Cc^3(9Aa^p+Bb^p)^3}{AB^3(ab^3)^p}=3^{3} \frac{(Aa^p+Bb^p)(9Aa^p+Bb^p)^3}{AB^3(ab^3)^p}=3^{3}\frac{(1+\mu(E))(9+\mu(E))^3}{\mu(E)^3}$ for $\mu(E)= \frac{Bb^p}{Aa^p}$.
		For each $i=1,2,\ldots,s$, the equation $j_E=j_i$ has at most four solutions in $K$. So, there exists 
		$\mu_1, \mu_2, ..., \mu_n \in K$ with $n\leq 4s$ such that $E$ is modular for all $\mu(E) \notin\{\mu_1, \mu_2, ..., \mu_n\}$.
		If $\mu(E)= \mu_k$ for some $k \in \{1, 2, \ldots, n \}$, then $\left(\frac{b}{a} \right)^p=\frac{A\mu_k}{B}$.
		This equation determines $p$ uniquely, denoting it $p_k$. Suppose $q \neq l$ are primes such that $\left(\frac{b}{a} \right)^q=\left(\frac{b}{a} \right)^l$, which means $\left(\frac{b}{a}\right)$ is a root of unity. Since $K$ is totally real, we get $b=\pm a$. By Remark~\ref{remark for W_K and W_K'}, this cannot happen if we choose $p>  v_\mfP (C)$ for $\mfP \in S_K$.
	    Now, the proof of the theorem follows by taking $D=\max \big\{p_1,...,p_m,  v_\mfP (C) \big\}$.
	\end{proof}

	\subsection{Reduction type of Frey elliptic curves}
	The following lemma characterizes the type of reduction of the Frey curve $E:= E_{a,b,c}$ 
	at primes $\mfq$ away from $S_K^{\prime}$.
	\begin{lem}
		\label{reduction away from S}
		Let $(a,b,c) \in \mcO_K^3$ be a non-trivial primitive solution to the equation~\eqref{Ax^p+By^p=Cz^3} with exponent $p$, and let $E$ be the associated Frey curve. Then, at all primes $\mfq$ away from $ S_K^{\prime}$, $E$ is minimal, semi-stable at $\mfq$ and satisfies $p | v_\mfq(\Delta_E)$. Let $\mfn$ be the conductor of $E$ and $\mfn_p$ be as in \eqref{conductor of elliptic curve}. Then,
			\begin{equation}
				\label{conductor of E and E' x^2=By^p+Cz^p Type I}
				\mfn=\prod_{\mfP \in S_K^{\prime}}\mfP^{r_\mfP} \prod_{\mfq|ab,\ \mfq \notin S_K^{\prime}}\mfq,\ \mfn_p=\prod_{\mfP \in S_K^{\prime}}\mfP^{r_\mfP^{\prime}},
			\end{equation}
			where $0 \leq r_\mfP^{\prime} \leq r_\mfP $ with $r_\mfP 
			\leq 2+3v_\mfP(3)$ for $\mfP |3$ and $ r_\mfP\leq 2+6v_\mfP(2)$ for $\mfP \nmid 3$.
\end{lem}
	
	\begin{proof}
	Let $\mfq \in P \setminus S_K^\prime$. 
\begin{itemize}
	\item If $\mfq \not|\Delta_E$, then $E$ has good reduction at $\mfq$
	and $p | v_\mfq(\Delta_E)=0$.
	\item If $\mfq|\Delta_E=3^3AB^3C^8(ab^3)^p$, then $\mfq$ exactly divides one of $a$ and $b$, since $(a,b,c)$ is primitive and $\mfq \nmid 3ABC$. This implies $\mfq \nmid c$, hence $\mfq\nmid c_4=3^2C^3c(9Aa^p+Bb^p)$. Therefore, $E$ is minimal, and $E$ has multiplicative reduction at $\mfq$.
\end{itemize}
        Since $v_\mfq(\Delta_E)=p v_\mfq(ab^3) $, $p | v_\mfq(\Delta_E)$. 
        By the definition of $\mfn_p$ in~\eqref{conductor of elliptic curve}, we get $\mfq \nmid \mfn_p$ for all $\mfq \notin S_K^{\prime}$. Finally, for $\mfP \in S_K^{\prime}$, the bounds on $r_\mfP$ follow from \cite[Theorem IV.10.4]{S94}.
	\end{proof}
%

	\subsubsection{Type of reduction with image of inertia}
    Now, we recall ~\cite[Lemma 3.4]{FS15}, which will be useful for the types of reduction of the Frey curve at  $\mfq \in P$. 
	\begin{lem}
		\label{criteria for potentially multiplicative reduction}
		Let $E/K$ be an elliptic curve and $p>5$ be a prime. For $\mfq \in P$ with $\mfq \nmid p$, $E$ has potentially multiplicative reduction at $\mfq$ and $p \nmid v_\mfq(j_E)$ if and only if $p | \# \bar{\rho}_{E,p}(I_\mfq)$.
	\end{lem}
   The following lemma determines the type of reduction of the Frey curve  $E$ at primes $\mfq \nmid 3pABC$.
	\begin{lem}
		\label{Type of reduction at q away from 2,p,B x^2=By^p+Cz^p}
		Let $(a,b,c)\in \mcO_K^3$ be a non-trivial primitive solution to the equation~\eqref{Ax^p+By^p=Cz^3} with exponent $p>5$, and let $E$ be the associated Frey curve. Suppose $\mfq \in P$ with $\mfq\nmid 3pABC$. Then $p \nmid \#\bar{\rho}_{E,p}(I_\mfq)$.
 	\end{lem}
	
	\begin{proof}
		By Lemma~\ref{criteria for potentially multiplicative reduction}, it is enough to show that
		either $v_\mfq(j_E) \geq 0$ or $p | v_\mfq(j_E)$. Recall that $\Delta_E=3^3AB^3C^8(ab^3)^p$ and $ c_4=3^2C^3c(9Aa^p+Bb^p)$.
		\begin{itemize}
			\item 	If $\mfq \nmid \Delta_E$, then $E$ has good reduction at $\mfq$, and hence $v_\mfq(j_E)\geq 0$.
			\item    If $\mfq | \Delta_E$ then $\mfq |ab$,
			and hence $\mfq$ divides exactly one of $a$ and $b$. Therefore, $\mfq \nmid c_4$ and $p |v_\mfq(j_E)=-p
			v_\mfq(ab^3)$.
		\end{itemize} 
	Hence, we are done with the proof of the lemma.
	\end{proof}
	We will now discuss the type of reduction of $E_{a,b,c}$ at  $\mfP \in S_K$ with $(a,b,c) \in W_K$.
	\begin{lem}
		\label{reduction on T and S}
	Let $\mfP \in S_K$. Let $(a,b,c)\in W_K$ be a solution to~\eqref{Ax^p+By^p=Cz^3} with exponent
		$ p > \max \left\{ 3v_\mfP(3)+ v_\mfP(ABC), |3 v_\mfP(3)\pm v_\mfP(AB^{- 1})| \right\}.$ Let $E:=E_{a,b,c}$ be the associated Frey curve. Then $\ v_\mfP(j_E) < 0$ and $p \nmid v_\mfP(j_E)$, equivalently $p | \#\bar{\rho}_{E,p}(I_\mfP)$.
    \end{lem} 
    \begin{proof}
    Recall that $ j_E=3^{3} \frac{Cc^3(9Aa^p+Bb^p)^3}{AB^3(ab^3)^p}=3^{3} \frac{(Aa^p+Bb^p)(9Aa^p+Bb^p)^3}{AB^3(ab^3)^p}.$ Since $p >v_\mfP(C)$, by Remark~\ref{remark for W_K and W_K'}, we have $\mfP|a$ or $\mfP|b$ but not both.
%
    	\begin{itemize}
    	\item If $\mfP |a$, then $\mfP \nmid b$. 
    	Since $p > v_\mfP(B)$, $v_\mfP(j_E)= 3v_\mfP(3)+v_\mfP(B)+ 3v_\mfP(B)-v_\mfP(AB^3)-pv_\mfP(a) =  3v_\mfP(3)+v_\mfP(BA^{-1})-pv_\mfP(a)$. Since $p > |3v_\mfP(3)+v_\mfP(BA^{-1})|$, $v_\mfP(j_E) <0$ and $p \nmid v_\mfP(j_E)$.
    	\item If $\mfP |b$, then $\mfP \nmid a$. 
    		Since $p > 3v_\mfP(3)+ v_\mfP(A)$, $v_\mfP(j_E)= 3v_\mfP(3)+v_\mfP(A)+6v_\mfP(3)+ 3v_\mfP(A)-v_\mfP(AB^3)-3pv_\mfP(b)=3 \left( 3v_\mfP(3)+v_\mfP(AB^{-1})-pv_\mfP(b) \right)$. Since $p > |3v_\mfP(3)+v_\mfP(AB^{-1})|$, $v_\mfP(j_E) <0$ and $p \nmid v_\mfP(j_E)$.  
    \end{itemize}
    		Hence, by Lemma~\ref{criteria for potentially multiplicative reduction}, we get $p | \#\bar{\rho}_{E,p}(I_\mfP)$.
	\end{proof}

%
%


	\subsection{Proof of Theorem~\ref{main result1 for x^2=By^p+Cz^p Type I}.}
%
The proof of this theorem depends on the following auxiliary result:
	\begin{thm}
		\label{auxilary result x^2=By^p+Cz^p over W_K}
		Let $K$ be a totally real field. Then, there is a constant $V=V_{K,A,B,C}>0$ (depending on $K,A,B,C$) such that the following hold.
	    Let $(a,b,c)\in W_K$ be a solution to the equation \eqref{Ax^p+By^p=Cz^3} with exponent $p>V$, and let $E$ be the Frey curve as in \eqref{Frey curve for x^2=By^p+Cz^p of Type I}. Then, there exists an elliptic curve $E^\prime/K$ such that:
		\begin{enumerate}
			\item $E^\prime/K$ has good reduction away from $S_K^{\prime}$ and has a non-trivial $3$-torsion point;
			\item $\bar{\rho}_{E,p} \sim\bar{\rho}_{E^\prime,p}$, and  $v_\mfP(j_{E^\prime})<0$ for $\mfP \in S_K$.
		\end{enumerate}
	\end{thm}

    \begin{proof}[Proof of Theorem~\ref{auxilary result x^2=By^p+Cz^p over W_K}] 
		By Theorem~\ref{modularity of Frey curve x^2=By^p+Cz^p over W_K}, $E$ is modular for primes $p>D:=D_{K,A,B,C}$ with $D \gg 0$. By Lemma~\ref{reduction away from S}, $E$ is semi-stable away from $S_K^{\prime}$. If necessary, we can take the Galois closure of $K$ to ensure that $\bar{\rho}_{E,p}$ is irreducible for $p \gg 0$ (cf.~\cite[Theorem 2]{FS15 Irred}).
		
		By Theorem~\ref{level lowering of mod $p$ repr}, there exists a Hilbert
        modular newform $f$ of parallel weight $2$, level $\mfn_p$ and some prime $\omega$ of $\Q_f$ such that $\omega | p$ and $\bar{\rho}_{E,p} \sim \bar{\rho}_{f,\omega}$ for $p \gg 0$.
By allowing $p$ to be sufficiently large, we can assume  $\Q_f=\Q$. This step uses standard ideas originally due to Mazur that can be found in~\cite[\S 4]{BS04}, \cite[Proposition 15.4.2]{C07}, and~\cite[\S 4]{FS15}.

		Let $\mfP \in S_K$. Then $E$ has potential multiplicative reduction at $\mfP$ and $p | \#\bar{\rho}_{E,p}(I_\mfP)$ for $p \gg 0$
		(cf. Lemma~\ref{reduction on T and S}). The existence of $E_f$ then follows from 
		\cite[Corollary 2.2]{FS15} 
		for all $p\gg 0$ after leaving primes $p$ with $p \mid \left( \text{Norm}(K/\Q)(\mfP) \pm 1 \right)$. Therefore, $\bar{\rho}_{E,p} \sim \bar{\rho}_{E_f,p}$ for some elliptic curve $E_f$ with conductor $\mfn_p$ for $p>V=V_{K,A,B,C}$, where $V_{K,A,B,C}$ is the maximum of all the above implicit/explicit lower bounds.
		
		\begin{itemize}
			\item Since the conductor of $E_f$ is $\mfn_p$ given in \eqref{conductor of E and E' x^2=By^p+Cz^p Type I}, $E_f$ has good reduction away from $S_K^{\prime}$. Now, arguing as in~\cite[page 1247]{M22}, we can enlarge the constant $V$ and by possibly replacing $E_f$ with an isogenous curve, say $E^\prime$, we get $E^\prime/ K$ has a non-trivial $3$-torsion point. Since $E_f \sim E^\prime$, $E^\prime$ has good reduction away from $S_K^{\prime}$. 

			\item Since $E_f$ is isogenous to $E^\prime$ and $\bar{\rho}_{E,p} \sim \bar{\rho}_{E_f,p}$ implies $\bar{\rho}_{E,p} \sim \bar{\rho}_{E^\prime,p}$. As a result, we obtain $p | \# \bar{\rho}_{E,p}(I_\mfP)=\# \bar{\rho}_{E^\prime,p}(I_\mfP)$ for any $\mfP \in S_K$. Finally, by Lemma~\ref{criteria for potentially multiplicative reduction}, we have $v_\mfP(j_{E^\prime})<0$ for any $\mfP \in S_K$.
			\end{itemize}
		This completes the proof of the theorem.
	\end{proof}
	We now prove Theorem~\ref{main result1 for x^2=By^p+Cz^p Type I}, and its inspired from that of~\cite[Theorem 3]{M22}.
	\begin{proof}[Proof of Theorem~\ref{main result1 for x^2=By^p+Cz^p Type I}]
		Suppose $(a,b,c)\in W_K$ is a solution to the equation~\eqref{Ax^p+By^p=Cz^3} with exponent $p>V$, where $V=V_{K,A,B,C}$ be the constant as in Theorem~\ref{auxilary result x^2=By^p+Cz^p over W_K}. By Theorem~\ref{auxilary result x^2=By^p+Cz^p over W_K}, there exists an elliptic curve $E^\prime/K$ having a non-trivial $3$-torsion point and good reduction away from $S_K^{\prime}$.
Then the elliptic curve $E^\prime/K$ has a model of the form
		\begin{equation}
			\label{j invariant of E'}
		E^\prime:y^2+a'xy+b'y=x^3
		\end{equation}		
	for some $a',b' \in K$ with $j$-invariant $j_{E^\prime}= \frac{a'^3(a'^3-24b')^3}{b'^3(a'^3-27b')}$. Since $E^\prime$ has good reduction away from $S_K^{\prime}$, $j_{E^\prime} \in \mcO_{S_K^{\prime}}$. 
	
	Take $\lambda := \frac{a'^3}{b'}$ and $\mu := \lambda-27$. Then $\lambda \in \mcO_{S_K^{\prime}}$ and $ \mu \in \mcO_{S_K^{\prime}}^\ast$
	(cf.~\cite[Lemma 16(ii)]{M22}). By~\cite[Lemma-17(ii)]{M22},
	we get $ \lambda  \mcO_K=I^3J$ for some fractional ideal $I$ and $S_K^{\prime}$-ideal $J$. Since $J$ is $S_K^{\prime}$-ideal,  $1= [I]^3 \in \Cl_{S_K^{\prime}}(K)$. By hypothesis $\Cl_{S_K^{\prime}}(K)[3]=1$ which gives $I=\gamma I_1$ for some
	$\gamma \in \mcO_K$ and $S_K^{\prime}$-ideal $I_1$. Thus, $\lambda \mcO_K=\gamma^3I_1^3J$ and hence $(\frac{\lambda}{\gamma^3}) \mcO_K$ is an $S_K^{\prime}$-ideal. Therefore, $u=\frac{\lambda}{\gamma^3} \in \mcO_{S_K^{\prime}}^\ast$.
	Now, divide the equation $\mu +27=\lambda$ by $u$ to obtain 
	$\alpha +\beta =\gamma^3$, where $\alpha= \frac{\mu}{u} \in \mcO_{S_K^{\prime}}^\ast$ and $\beta =\frac{27}{u} \in \mcO_{S_K^{\prime}}^\ast$, which implies $ \alpha \beta^{-1}=\frac{\mu}{27}$.
	By~\eqref{assumption for main result1 for x^2=By^p+Cz^p Type I}, there exists $\mfP \in S_K$ with $|v_\mfP(\alpha \beta^{-1})|= |v_\mfP(\frac{\mu}{27} )| \leq 3 v_\mfP(3)$. This means
	\begin{equation}
		\label{inequality for valution of mu}
		0 \leq v_\mfP(\mu) \leq 6 v_\mfP(3).
	\end{equation}
    We now show that the bounds on $v_\mfP(\mu)$ would imply that $v_\mfP(j_{E^\prime}) \geq 0$. 
	Write $j_{E^\prime}$ in terms of $\mu$ yields $j_{E^\prime}= \frac{(\mu+27)(\mu+3)^3}{\mu}$,
	which means 
	\begin{equation}
		\label{j inv in terms of mu}
		v_\mfP(j_{E^\prime})=v_\mfP(\mu+27)+3 v_\mfP(\mu +3)-v_\mfP(\mu).
	\end{equation} 
\begin{itemize}
 \item If $0 \leq v_\mfP(\mu) \leq v_\mfP(3)$, then $v_\mfP(\mu+27)=v_\mfP(\mu)$ and $v_\mfP(\mu +3) \geq v_\mfP(\mu)$. 
       By \eqref{j inv in terms of mu}, we get $v_\mfP(j_{E^\prime})=3v_\mfP(\mu)\geq 0$.
 \item  If $  v_\mfP(3) < v_\mfP(\mu) \leq 3v_\mfP(3)$, then $v_\mfP(\mu+27)\geq v_\mfP(\mu)$ and $ v_\mfP(\mu +3)= v_\mfP(3)$. By \eqref{j inv in terms of mu}, we have  $v_\mfP(j_{E^\prime}) >0$.
 \item  If $  3v_\mfP(3) < v_\mfP(\mu) \leq 6v_\mfP(3)$, then $v_\mfP(\mu+27)= 3v_\mfP(3)$ and $ v_\mfP(\mu +3)= v_\mfP(3)$.
        By \eqref{j inv in terms of mu}, we have $v_\mfP(j_{E^\prime})=6v_\mfP(3)- v_\mfP(\mu) \geq 0$. 
\end{itemize}
In all cases, we get $v_\mfP(j_{E^\prime}) \geq 0$, which is a contradiction to Theorem~\ref{auxilary result x^2=By^p+Cz^p over W_K}. This completes the proof of the theorem.
\end{proof}
Now, we are in a position to prove Proposition~\ref{main result2 for x^2=By^p+Cz^p Type I}, and its proof is inspired from that of \cite[Theorem 11]{M22}.
\begin{proof}[Proof of Proposition~\ref{main result2 for x^2=By^p+Cz^p Type I}]
Let $S_K= \{\mfP\}$. By Theorem~\ref{main result1 for x^2=By^p+Cz^p Type I}, it suffices to show that for every solution $(\alpha, \beta, \gamma)\in  \mcO_{S_K}^\ast \times \mcO_{S_K}^\ast \times \mcO_{S_K}$ to the equation $\alpha+\beta=\gamma^3$, $\mfP$ satisfies $|v_\mfP(\alpha\beta^{-1}) |\leq 3v_\mfP(3)$. If necessary, by scaling cubic powers of $\mfP$ and swapping $\alpha, \beta$, we can assume $0\leq v_\mfP(\beta)\leq v_\mfP(\alpha)$ with $v_\mfP(\beta)=0$ or $1$ or $2$. 
	\begin{enumerate}
		\item Suppose $v_\mfP(\beta)= 1$ or $2$. If $v_\mfP(\alpha)>v_\mfP(\beta)$, then $v_\mfP(\gamma^3)=v_\mfP(\alpha+\beta)=v_\mfP(\beta)$, which cannot happen since $v_\mfP(\gamma^3)$ is a multiple of $3$. So, $v_\mfP(\alpha)=v_\mfP(\beta)$. Hence, $\left|v_\mfP(\alpha \beta^{-1}) \right|=0 <3v_\mfP(3)$.

		\item Suppose $v_\mfP(\beta)=0$. Then $\beta \in \mcO_K^\ast$.
		\begin{itemize}
		 \item If $\beta$ is a cube, then divide the equation $\alpha +\beta= \gamma^3$ by $\beta$ to obtain an equation of the form $\alpha'+1=\gamma'^3$, where $\alpha'=  \alpha \beta^{-1} \in \mcO_{S_K^{\prime}}^\ast$ with $v_\mfP(\alpha')\geq 0$, and $\gamma' \in \mcO_{S_K^{\prime}}$. By~\eqref{assumption for main result2 for x^2=By^p+Cz^p Type I}, we obtain $|v_\mfP(\alpha \beta^{-1}) |\leq 3v_\mfP(3)$.

		 \item Suppose $\beta$ is not a cube. If $v_\mfP(\alpha)\leq 3v_\mfP(3)$, then we are done. Otherwise, $v_\mfP(\alpha)> 3v_\mfP(3)>1$. This gives $\alpha \equiv 0 \pmod {3^3}$ and $\gamma^3=\alpha+\beta \equiv \beta\pmod {3^3}$.
		 Since $v_\mfP(\gamma^3)= v_\mfP(\alpha+\beta)=0$ and $S_K=S_K^\prime$, we get $\gamma \in \mcO_K$. The field $L= K(\zeta_3, \beta^\frac{1}{3})$ is a degree $3$ extension of $K(\zeta_3)$.		 We will now show that $L$ is unramified at $3$ to get a contradiction to $3 \nmid  h_{K(\zeta_3)}$. Consider an element $\theta:= \frac{\gamma^2+\gamma \zeta_3 \beta^\frac{1}{3}+ \zeta_3^2 \beta^\frac{1}{3}}{3}$.
         The minimal polynomial of $\theta$ is $m_\theta(x)= x^3+\frac{\gamma( \gamma^3-\beta)}{3}x^2- \gamma^2 x - \frac{(\gamma^3 - \beta)^2}{27}$. Then $m_\theta(x) \in \mcO_K[x]$ with discriminant $\Delta_\theta=-\frac{2\gamma^3( \gamma^3-\beta)^3}{3^5}-\frac{4\gamma^3( \gamma^3-\beta)^5}{3^9}+\frac{\gamma^6( \gamma^3-\beta)^2}{3^2}-4\gamma^6-\frac{( \gamma^3-\beta)^4}{3^3}$. 
		 Since $\Delta_\theta \equiv -4\gamma^6 \pmod 3$ and $v_\mfP(\gamma^3)=0$, $L$ is unramified at $3$, which contradicts our hypothesis that $3 \nmid h_{K(\zeta_3)}$.
        \end{itemize}
    \end{enumerate}
This completes the proof of the proposition.
\end{proof}

\section{Solutions of $Ax^p+By^p=Cz^3$ over $K$}
\label{section for $x^2=By^p+2^rz^p$ and $2x^2=By^p+2^rz^p$ over $K$} 
In this section, we shall examine the $K$-solutions of the equation
\begin{equation}
	\label{Ax^p+By^p=qz^3}
	Ax^p+By^p=Cz^3
\end{equation}
with exponent $p$, where $A,B,C \in \mcO_K \setminus\{0\}$. Throughout this section, we assume $C \in \mcO_K^*$ or $C=uq$  with $u \in \mcO_K^\ast$ and $q \in \mathbb{P} \setminus \{3\}$. In both cases, $v_\mfP(C)=0$. Recall that $S_K^{\prime}=\{\mfP \in P : \mfP |3ABC\}$.

\subsection{Main result}
\label{section for main result over K}
We write $(ES)$ for ``either $[K: \Q]\equiv 1 \pmod 2$ or Conjecture \ref{ES conj} holds for $K$". 
Let $S \subseteq \mcO^3_K$ be the set of all solutions of \eqref{Ax^p+By^p=qz^3} of the form $\left\{ (u, \pm u, c) : u \in \mcO_K^*,\ c \in \mcO_K\setminus \{0\} \right\}.$
We now show that the equation \eqref{Ax^p+By^p=qz^3} with exponent $p$ has no asymptotic solution in $\mcO_K^3  \setminus S$. More precisely,
\begin{thm}
	\label{main result1 for Ax^p+By^p=qz^3 over K}
	Let $K$ be a totally real field satisfying $(ES)$ with $\Cl_{S_K^{\prime}}(K)[3]=1$. Assume $3$ is inert in $K$ and let $S_K= \{\mfP\}$. Suppose for every solution $(\alpha, \beta, \gamma) \in \mcO_{S_K^{\prime}}^\ast \times \mcO_{S_K^{\prime}}^\ast \times \mcO_{S_K^{\prime}}$ to $\alpha+\beta=\gamma^3$, $\mfP$ satisfies
	\begin{equation}
		\label{assumption for main result1 x^2=By^p+2^rz^p over K}
		v_\mfP( \alpha \beta^{-1}) =2.
	\end{equation}
	Further, if $v_\mfP(A)=1$ and $ v_\mfP(B) \in \{0,2\}$, then the equation $Ax^p+By^p=Cz^3$ with exponent $p$ has no asymptotic solution in $\mcO_K^3  \setminus S$.
\end{thm}	
\subsection{Modularity of Frey elliptic curves}
We now prove the modularity of the Frey curve $E:=E_{a,b,c}$ in \eqref{Frey curve for x^2=By^p+Cz^p of Type I} associated to any non-trivial primitive solution $(a,b,c)\in$ $\mcO_K^3 \setminus S$ for primes $p \gg 0$.
\begin{thm}
	\label{modularity of Frey curve of x^2=By^p+2^rz^p over K}
		Let K be a totally real number field. Then, there exists a constant $D:=D_{K,A,B,C}$ (depending on $K,A,B,C$) such that
	for any non-trivial primitive solution $(a,b,c)\in \mcO_K^3 \setminus S$ to the equation~\eqref{Ax^p+By^p=qz^3} with exponent $p>D$, the Frey curve $E:=E_{a,b,c}$ given in~\eqref{Frey curve for x^2=By^p+Cz^p of Type I} is modular.
\end{thm}

\begin{proof}
	Arguing as in the proof of Theorem~\ref{modularity of Frey curve x^2=By^p+Cz^p over W_K}, there exists $\mu_k \in K$ with $1 \leq k \leq n$ such that $E/K$ is modular for all $\mu(E) \notin\{\mu_1, \mu_2, ..., \mu_n\}$.
	If $\mu= \mu_k$ for some $k \in \{1, 2, \ldots, n \}$, then $\left(\frac{b}{a} \right)^p= \frac{A\mu_k}{B}$.
	The above equation determines $p$ uniquely, denoting it $p_k$. Otherwise, we get $b=\pm a$. Now, we will show that $a= u,\ b= \pm u$ for some $u \in \mcO_K^\ast.$
	
	Let $\mfq \in P$. If $\mfq |a$, then $\mfq^p|Aa^p+ Bb^p=Cc^3$.
If $C \in \mcO_K^\ast$, then $\mfq | c$.
If $C=uq$ with $u \in \mcO_K^\ast$ and $q \in \mathbb{P}$, then $\mfq^p|qc^3$. Taking $p> [K: \Q]$, we have $\mfq|c$. In both cases, we get a contradiction to the primitivity of $(a,b,c)$. Therefore, $a  \in \mcO_K^\ast$. Hence $(a,b,c) \in S$, which is a contradiction. Arguing as in the proof of Theorem~\ref{modularity of Frey curve x^2=By^p+Cz^p over W_K}, the proof of the theorem follows by taking $D=\max \big\{ p_1,...,p_m, [K:\Q] \big\}$.
\end{proof}

	\subsection{Reduction type of Frey elliptic curves}
The following lemma will be useful for the reduction of the Frey curve $E_{a,b,c}$ with $(a,b,c)\in$ $\mcO_K^3$ at $\mfP \in S_K$.
	\begin{lem}
	\label{image of inertia comparision}
	Let $E/K$ be an elliptic curve and $p\geq 5$ be a prime. Assume $3$ is unramified in $K$. Suppose $E$ has potential good reduction at $\mfP$ for some $\mfP \in S_K$. 
	\begin{enumerate}
	\item If $ v_\mfP(\Delta_E)=4$ or $10$, then $ \#\bar{\rho}_{E,p}(I_\mfP)=3$ or $6$.
	\item  If $ \#\bar{\rho}_{E,p}(I_\mfP)=3$ or $6$, then $ v_\mfP(\Delta_E) \in \{4,6,10,12\}$.
\end{enumerate}
\begin{proof}
	This lemma is a special case of \cite[Corollaire to Th\'{e}or\'{e}me 1]{K90}.
\end{proof}

\end{lem}
The following lemma specifies the type of reduction of the Frey curve $E:=E_{a,b,c}$ given in ~\eqref{Frey curve for x^2=By^p+Cz^p of Type I} at $\mfP \in S_K$ when $(a,b,c)\in$ $\mcO_K^3$. More precisely,

\begin{lem}
	\label{reduction on T and S x^2=By^p+Cz^p Type II over K}
	 Assume $3$ is inert in $K$ and let $S_K= \{\mfP\}$. Let $(a,b,c)\in \mcO_K^3$ be a non-trivial primitive solution to the equation~\eqref{Ax^p+By^p=qz^3} with exponent
	 $$p > \max \left\{ 3v_\mfP(3)+ v_\mfP(AB), |3v_\mfP(3)\pm v_\mfP(AB^{-1})| \right\}.$$ 
	 Let $E$ be the associated Frey curve. If $v_\mfP(A)=1$ and $ v_\mfP(B) \in \{0,2\}$, then either $p | \#\bar{\rho}_{E,p}(I_\mfP)$ or $\#\bar{\rho}_{E,p}(I_\mfP) \in \{3,6\}$.  
\end{lem}

\begin{proof}
If $\mfP |ab$, then by Lemma~\ref{reduction on T and S}, we have  $p | \#\bar{\rho}_{E,p}(I_\mfP)$. 
Suppose $\mfP \nmid ab$. Now, recall that $ j_E=3^{3} \frac{Cc^3(9Aa^p+Bb^p)^3}{AB^3(ab^3)^p}=3^{3} \frac{(Aa^p+Bb^p)(9Aa^p+Bb^p)^3}{AB^3(ab^3)^p}.$ Then $v_\mfP(j_E)=3 +v_\mfP(Aa^p+Bb^p)+ 3v_\mfP (9Aa^p+Bb^p)- v_\mfP(AB^3)$. Now,
\begin{itemize}
 \item if $v_\mfP(B)=0$, then  $v_\mfP(j_E)= 3-v_\mfP(A)>0$,
 \item if $v_\mfP(B)=2$, then $v_\mfP(j_E)\geq 3 +v_\mfP(A)+ 3v_\mfP (B)- v_\mfP(AB^3) >0$. Hence, $E$ has potential good reduction at $\mfP$.
\end{itemize}
Since $\Delta_E=3^3AB^3C^8(ab^3)^p$, $v_\mfP(\Delta_E)=3+ v_\mfP(AB^3)$. The hypothesis on the valuations of $A,B$ implies $v_\mfP(AB^3)=1$ or $7$, hence $ v_\mfP(\Delta_E)=4$ or $10$. By Lemma~\ref{image of inertia comparision}, we have $ \#\bar{\rho}_{E,p}(I_\mfP)=3$ or $6$. So, we are done with the proof of the lemma.
%
%
\end{proof}

\subsection{Proof of Theorem~\ref{main result1 for Ax^p+By^p=qz^3 over K}}	
The proof of this theorem depends on the following result:
\begin{thm}
\label{auxilary result x^2=By^p+2^rz^p over K}
Let $K$ be a totally real field satisfying $(ES)$. Assume $3$ is inert in $K$, and $v_\mfP(A)=1$, $ v_\mfP(B) \in \{0,2\}$. 
Then, there is a constant $V=V_{K,A,B,C}>0$ (depending on $K,A,B,C$) such that the following hold. Let $(a,b,c)\in \mcO_K^3  \setminus S$ be a non-trivial primitive solution to the equation~\eqref{Ax^p+By^p=qz^3} with exponent $p >V$, and let $E$ be the Frey curve as in ~\eqref{Frey curve for x^2=By^p+Cz^p of Type I}. Then there exists an elliptic curve $E^\prime/K$ such that:

\begin{enumerate}
	\item $E^\prime/K$ has good reduction away from $S_K^{\prime}$ and has a non-trivial $3$-torsion point, and
	$\bar{\rho}_{E,p} \sim\bar{\rho}_{E^\prime,p}$;
	\item  For $\mfP \in S_K$, either $v_\mfP(j_{E^\prime})<0$ or $v_\mfP(j_{E^\prime}) \equiv 0 \text{ or } 2 \pmod 3$.
\end{enumerate}
\end{thm}
\begin{proof}
Arguing as in the proof of Theorem~\ref{auxilary result x^2=By^p+Cz^p over W_K}, the first part of Theorem~\ref{auxilary result x^2=By^p+2^rz^p over K} follows from ~\cite[Theorem 2]{FS15 Irred}, Theorem~\ref{modularity of Frey curve of x^2=By^p+2^rz^p over K}, Lemma~\ref{reduction away from S} and Theorem~\ref{level lowering of mod $p$ repr}. Let $\mfP \in S_K$ be the unique prime lying above $3$. If $p | \# \bar{\rho}_{E,p}(I_\mfP)= \# \bar{\rho}_{E^\prime,p}(I_\mfP)$, then by Lemma~\ref{criteria for potentially multiplicative reduction}, we get $v_\mfP(j_{E^\prime})<0$. If $p \nmid \# \bar{\rho}_{E,p}(I_\mfP)$, then by Lemma~\ref{reduction on T and S x^2=By^p+Cz^p Type II over K}, we conclude that $\#\bar{\rho}_{E^\prime,p}(I_\mfP)= \#\bar{\rho}_{E,p}(I_\mfP) \in \{3,6\}$. If $v_\mfP(j_{E^\prime}) < 0$, then we are done. If $v_\mfP(j_{E^\prime}) \geq 0$, then by Lemma~\ref{image of inertia comparision}, we have $v_\mfP(\Delta_{E^\prime}) \equiv 0 \text{ or } 1 \pmod 3$. Since $j_{E^\prime} = \frac{c_4^3}{\Delta_{E'}}$, $v_\mfP(j_{E^\prime}) \equiv - v_\mfP(\Delta_{E^\prime}) \pmod3$. Hence, $v_\mfP(j_{E^\prime}) \equiv 0 \text{ or } 2 \pmod 3$. This completes the proof of the theorem.
\end{proof}

\begin{proof}[Proof of Theorem~\ref{main result1 for Ax^p+By^p=qz^3 over K}]
Let $(a,b,c)\in \mcO_K^3  \setminus S$ be a non-trivial primitive solution to the equation~\eqref{Ax^p+By^p=qz^3} with exponent $p >V$, where $V=V_{K,A,B,C}$ be the constant as in Theorem~\ref{auxilary result x^2=By^p+2^rz^p over K}. By Theorem~\ref{auxilary result x^2=By^p+2^rz^p over K}, there exists an elliptic curve $E^\prime/K$ having a non-trivial $3$-torsion point and good reduction away from $S_K^{\prime}$. 
By ~\eqref{assumption for main result1 x^2=By^p+2^rz^p over K}, for every solution $(\alpha, \beta, \gamma) \in \mcO_{S_K^{\prime}}^\ast \times \mcO_{S_K^{\prime}}^\ast \times \mcO_{S_K^{\prime}}$ to $\alpha+\beta=\gamma^3$, the prime $\mfP$ satisfies $v_\mfP(\alpha\beta^{-1}) =2$. 

Now, arguing as in the proof of Theorem~\ref{main result1 for x^2=By^p+Cz^p Type II}, we find $v_\mfP(j_{E^\prime}) \geq 0$ by using $|v_\mfP(\alpha\beta^{-1}) |=2 < 3v_\mfP(3)=3$.
Recall that $j_{E^\prime} = \frac{(\mu+27)(\mu+3)^3}{\mu}$, where $\mu=27 \alpha\beta^{-1}$. This implies $v_\mfP(j_{E^\prime}) \equiv v_\mfP(\mu+27)-v_\mfP(\mu) \pmod3 $. 
Since $v_\mfP(\alpha\beta^{-1}) =2$, $v_\mfP(\mu)=5$. Hence
$v_\mfP(j_{E^\prime}) \equiv 3-5 \equiv 1 \pmod 3$.
Therefore, $v_\mfP(j_{E^\prime}) \geq0$ and $ v_\mfP(j_{E^\prime}) \equiv 1 \pmod 3$, which contradicts Theorem~\ref{auxilary result x^2=By^p+2^rz^p over K}. We are done.
\end{proof}
We conclude this section with the following proposition:
\begin{prop}
	\label{prop for S is empty}
If $A,B,C \in \Z \setminus \{0\}$, then in Theorem~\ref{main result1 for Ax^p+By^p=qz^3 over K}, we can take $S=\emptyset$.
\end{prop}
\begin{proof}
	In order to prove Theorem~\ref{main result1 for Ax^p+By^p=qz^3 over K} for $S=\emptyset$, it is enough to prove Theorem~\ref{modularity of Frey curve of x^2=By^p+2^rz^p over K} for $S=\emptyset$. 
	Arguing as in the proof of Theorem~\ref{modularity of Frey curve of x^2=By^p+2^rz^p over K} for $(a,b,c)\in \mcO_K^3$, there exists $\mu_k \in K$ with $1 \leq k \leq n$ such that $E/K$ is modular for all $\mu(E) \notin\{\mu_1, \mu_2, ..., \mu_n\}$, where $ j_E=3^{3}\frac{(1+\mu(E))(9+\mu(E))^3}{\mu(E)^3}$ for $\mu(E)= \frac{Bb^p}{Aa^p}$. Without loss of generality, we can assume $\mu_1, \mu_2, ..., \mu_n \notin \Q^\ast$, since elliptic curves over $\Q$ are modular. If $\mu= \mu_k$ for some $k \in \{1, 2, \ldots, n \}$, then $\left(\frac{b}{a} \right)^p= \frac{A\mu_k}{B}$.
	The above equation determines $p$ uniquely; we denote it $p_k$. If not, we get $\frac{b}{a} =\pm1$ and hence $\mu_k=\pm\frac{B}{A} \in \Q^\ast$, which is a contradiction. Now, argue as in Theorem~\ref{modularity of Frey curve of x^2=By^p+2^rz^p over K} to complete the proof of the proposition.
\end{proof}

\section{Local criteria of $K$}
\label{section for loc criteria}
In this section, we present several local criteria of $K$ which imply Theorem~\ref{main result1 for x^2=By^p+Cz^p Type I}. First, we look at the case when $K$ is a quadratic field.
\begin{prop}[Quadratic]
	\label{loc crit for real quadratic field over W_K}
Let $d \geq 2$ be a square-free integer satisfying $d \equiv 2 \pmod3$, and let  $K=\Q(\sqrt{d})$. Assume $3 \nmid h_K h_{K(\zeta_3)}$. If $A,B,C \in \{u3^r \ | \ u \in \mcO_K^\ast,\ r \in \Z_{\geq 0}\}$, then the conclusion of Theorem~\ref{main result1 for x^2=By^p+Cz^p Type I} holds over $K$.
\end{prop}

\begin{proof}
	Since $d \equiv 2 \pmod 3$, $3$ is inert in $K$. Let $S_K= \{\mfP\}$. The hypotheses on $A,B,C$ imply that $S_K^\prime =S_K$. Now, arguing as in the proof of \cite[Theorem 12]{M22}, we see that the hypothesis of Proposition~\ref{main result2 for x^2=By^p+Cz^p Type I} is satisfied. Hence, we are done with the proof of proposition.
\end{proof}
 Now, we look at the case when $K$ is of odd degree.
\begin{prop}[Odd degree]
	\label{loc crit for odd degree field over W_K}
	Let $K$ be a field such that $3 \nmid h_K h_{K(\zeta_3)}$ and $n =[K:\Q]$.
	Suppose
	\begin{enumerate}
        \item $q \geq5$ be a rational prime with $\gcd(n, q-1)=1$ and $q$ totally ramifies in $K$,
		\item $3$ is either inert or $3=\mfP^n$ for some principal ideal $\mfP \in P$.
	\end{enumerate}
If $A,B,C \in \{u3^r \ | \ u \in \mcO_K^\ast,\ r \in \Z_{\geq 0}\}$, then the conclusion of Theorem~\ref{main result1 for x^2=By^p+Cz^p Type I} holds over $K$.
\end{prop}
\begin{proof}
Let $\mfP \in S_K$ be the unique prime ideal lying above $3$. By assumption, $S_K$ is principal. The hypotheses on $A,B,C$ imply that $S_K^\prime =S_K$. Now, arguing as in the proof of~\cite[Theorem 13]{M22}, we see that the hypothesis of Proposition~\ref{main result2 for x^2=By^p+Cz^p Type I} is satisfied. Hence, we are done with the proof of proposition.
\end{proof}
  \section*{Acknowledgments}
 The authors are
grateful to the anonymous referee for the mathematical suggestions and
comments which improved the article.   The authors express their sincere gratitude to Prof. Alain Kraus for his help in understanding~\cite{K90}. The first author thanks ISI Delhi
  for their hospitality during the preparation of this article.


\begin{thebibliography}{abc9999}
		\bibitem[BS04]{BS04}
		Bennett, Michael A.; Skinner, Chris M.
		Ternary Diophantine equations via Galois representations and modular forms. Canad. J. Math. 56 (2004), no. 1, 23--54.
		
		\bibitem[BVY04]{BVY04}
		Bennett, Michael A.; Vatsal, Vinayak; Yazdani, Soroosh. Ternary Diophantine equations of signature $(p,p,3)$. Compos. Math. 140 (2004), no. 6, 1399--1416.
		
 		\bibitem[Coh07]{C07}
		Cohen, Henri.
 		Number theory. Vol. II. Analytic and modern tools. Graduate Texts in Mathematics, 240. Springer, New York, 2007.
 		
 		\bibitem[Dar04]{D04}
 		Darmon, Henri.
 		Rational points on modular elliptic curves. CBMS Regional Conference Series in Mathematics, 101. Published for the Conference Board of the Mathematical Sciences, Washington, DC; by the American Mathematical Society, Providence, RI, 2004. 
		
		\bibitem[DG95]{DG95}
    	Darmon, Henri; Granville, Andrew. On the equations $z^m=F(x,y)$ and $Ax^p+By^q=Cz^r$. Bull. London Math. Soc. 27 (1995), no. 6, 513--543.
	
	
\bibitem[DM97]{DM97} 
Darmon, Henri; Merel, Lo\"{i}c.
Winding quotients and some variants of Fermat's last theorem. J. Reine Angew. Math. 490 (1997), 81--100.
	
	
\bibitem[Dec16]{D16}
Deconinck, Heline.
On the generalized Fermat equation over totally real fields. Acta Arith. 173 (2016), no. 3, 225--237.

	
	
%
	
	
	
	
	
	
	
	\bibitem[FLHS15]{FLHS15}
	Freitas, Nuno; Le Hung, Bao V.; Siksek, Samir.
	Elliptic curves over real quadratic fields are modular. Invent. Math. 201 (2015), no. 1, 159--206.
	
	
	
	\bibitem[FS15a]{FS15}
	Freitas, Nuno; Siksek, Samir.
	The asymptotic Fermat's last theorem for five-sixths of real quadratic fields. Compos. Math. 151 (2015), no. 8, 1395--1415.
	
	
	
	\bibitem[FS15b]{FS15 Irred}
	Freitas, Nuno; Siksek, Samir.
	Criteria for irreducibility of mod $p$ representations of Frey curves.
	J. Th\'{e}or. Nombres Bordeaux 27 (2015), no. 1, 67--76.
	
	
	
%
	
	
	
	
	
	

	
	\bibitem[IKO20]{IKO20}
	I\c{s}ik, Erman; Kara, Yasemin; \"Ozman, Ekin.
	On ternary Diophantine equations of signature $(p,p,2)$ over number fields.	Turkish J. Math. 44 (2020), no. 4, 1197--1211.
	
		\bibitem[IKO23]{IKO23}
	Isik, Erman; Kara, Yasemin; \"Ozman, Ekin. On ternary Diophantine equations of signature $(p,p,3)$ over number fields. Canad. J. Math. 75 (2023), no. 4, 1293--1313.
	
	\bibitem[Ivo03]{I03}
	Ivorra, Wilfrid.
	Sur les \'equations $x^p+2^\beta y^p=z^2$ et $x^p+2^\beta y^p=2z^2$. (French) [[On the equations $x^p+2^\beta y^p=z^2$ and $x^p+2^\beta y^p=2z^2$]] Acta Arith. 108 (2003), no. 4, 327--338.
	

	
	
	
	
	
	
%
	
	
	
	
	\bibitem[KS24a]{KS23}
	Kumar, Narasimha; Sahoo, Satyabrat.
    On the solutions of $x^p+y^p=2^rz^p$, $x^p+y^p=z^2$ over totally real fields. Acta Arith. 212 (2024), no. 1, 31--47.
    
    \bibitem[KS24b]{KS}
    Kumar, Narasimha; Sahoo, Satyabrat. 
    On the solutions of $x^2=By^p+Cz^p$ and $2x^2=By^p+Cz^p$ over totally real fields. Ramanujan J. 65 (2024), no. 1, 27--43.
    
    \bibitem[Kra90]{K90}
    Kraus, Alain. Sur le d\'{e}faut de semi-stabilit\'{e} des courbes elliptiques \`{a} r\'{e}duction additive. (French) [[On the failure of semistability of elliptic curves with additive reduction]] Manuscripta Math. 69 (1990), no. 4, 353--385.


	\bibitem[Maz78]{M78}
Mazur, B. Rational isogenies of prime degree (with an appendix by D. Goldfeld). Invent. Math. 44 (1978), no. 2, 129--162.

	\bibitem[Moc22]{M22}
	Mocanu, Diana.
	Asymptotic Fermat for signatures $(p,p,2)$ and $(p,p,3)$ over totally real fields. Mathematika 68 (2022), no. 4, 1233--1257.
	
	
	

	
	
	
	
	\bibitem[Rib90]{R90}
	Ribet, K. A.
	On modular representations of $\Gal (\bar \Q / \Q)$ arising from modular forms. Invent. Math. 100 (1990), no. 2, 431--476.
	
		
    
    \bibitem[Sik03]{S03}
    Siksek, Samir.
    On the Diophantine equation $x^2=y^p+2^kz^p$. J. Th\'eor. Nombres Bordeaux 15 (2003), no. 3, 839--846.
	
	\bibitem[Sil94]{S94}
	Silverman, Joseph H.
	Advanced topics in the arithmetic of elliptic curves. Graduate Texts in Mathematics, 151. Springer-Verlag, New York, 1994.
	
	
	
	
	
	
	
%
%
%
%
%
	 
	
	
	
	
	\bibitem[TW95]{TW95}
	Taylor, Richard; Wiles, Andrew. Ring-theoretic properties of certain Hecke algebras. Ann. of Math. (2) 141 (1995), no. 3, 553--572.
	
	\bibitem[Wil95]{W95}
	Wiles, Andrew.
	Modular elliptic curves and Fermat's last theorem. Ann. of Math. (2) 141 (1995), no. 3, 443--551.
	
	
	
	
	
	
	
	
	
	
	
\end{thebibliography}
\end{document}